\documentclass[11pt]{article}
\usepackage{fullpage, amsthm,amsmath,amssymb,tikz,enumerate}
\usetikzlibrary{decorations.markings}
\usepackage[affil-it]{authblk}
\usepackage{color}
\usepackage[font=footnotesize]{caption} 
\def\ch{\textrm{ch}}
\def\chio{\chi_o}

\newtheorem{thm}{Theorem}

\newtheorem{clm}{Claim}

\newtheorem{conj}{Conjecture}
\newtheorem{cor}[thm]{Corollary}
\theoremstyle{definition}

\RequirePackage{marginnote,hyperref}
\newcommand{\aside}[1]{\marginnote{\scriptsize{#1}}[0cm]}
\newcommand{\aaside}[2]{\marginnote{\scriptsize{#1}}[#2]}
\newcommand\Emph[1]{\emph{#1}\aside{#1}}

\def\less{\setminus}

\title{A Note on Odd Colorings of 1-Planar Graphs}
\author{Daniel W. Cranston\thanks{Department of Computer Science, Virginia
Commonwealth University, Richmond, VA, USA. E-mail: {\tt dcranston@vcu.edu}.} 
\hskip 1cm   Michael Lafferty\thanks{Department  of Mathematics, University of
Central Florida, Orlando, FL 32816, USA. Supported in part by NSF award
DMS-1854903. E-mail: {\tt Michael.Lafferty@ucf.edu}.} \hskip 1cm  Zi-Xia
Song\thanks{Department  of Mathematics, University of Central Florida, Orlando,
FL 32816, USA. Supported by  NSF award DMS-1854903. E-mail: {\tt Zixia.Song@ucf.edu}.}}

\date{}

\begin{document}
\maketitle
\begin{abstract}
A proper coloring of a graph is \emph{odd}  if every non-isolated vertex has
some color that appears an odd number of times on its neighborhood. This notion
was recently  introduced by  Petru\v{s}evski and \v{S}krekovski, who proved
that  every planar graph admits  an odd $9$-coloring;  they also conjectured
that every planar graph admits  an odd $5$-coloring.  Shortly after, this
conjecture was confirmed   for  planar graphs of girth at least seven by
Cranston; outerplanar graphs by Caro, Petru\v{s}evski, and \v{S}krekovski. 
Building on the work of  Caro, Petru\v{s}evski, and \v{S}krekovski, Petr and
Portier then further proved that  every planar graph admits an odd
$8$-coloring.  In this note we prove that every   1-planar graph admits an odd 
${23}$-coloring, where a graph is \emph{1-planar} if it can be drawn in the plane
so that each   edge is crossed by at most one other edge.
\end{abstract}


%---------------------------------------------------------------------------------------------------------
%---------------------------------------------------------------------------------------------------------
\baselineskip 16pt

\section{Introduction}

All  graphs considered in this paper are simple, finite, and undirected.  Let
$G$ be a graph.  We use $|G|$ and $e(G)$\aside{$|G|$, $e(G)$} to denote the number of vertices and
edges  in $G$, respectively. A proper coloring of $G$  is \Emph{odd}  if every
non-isolated vertex has some color that appears an odd number of times on its
neighborhood.  The \emph{odd chromatic number of $G$}, denoted
\Emph{$\chi_o(G)$}, is
the smallest number $k$ such that $G$ admits an odd $k$-coloring. 
Clearly, $\chio(G)\le |G|$, since we can simply color each vertex with its own
color.  The notion of an odd coloring was recently  introduced by
Petru\v{s}evski and \v{S}krekovski~\cite{PetSkr22}.  Using the discharging
method, they showed that $\chi_o(G)\leq 9$ for all planar graphs $G$.
Furthermore, they conjectured that all planar graphs admit an odd coloring
using $5$ colors. Fewer colors cannot suffice, since $\chi_o(C_5)=5$.

\begin{conj}[Petru\v{s}evski and \v{S}krekovski~\cite{PetSkr22}]\label{c:odd5}
Every planar graph admits an odd $5$-coloring.
\end{conj} 

Shortly after, Cranston~\cite{oddsparse} focused on the odd chromatic number of
sparse graphs, obtaining, for instance, $\chi_o(G)\leq 6$ for planar graphs $G$
of girth at least $6$ and $\chi_o(G)\leq 5$ for planar graphs $G$ of girth at
least $7$. Caro, Petru\v{s}evski, and \v{S}krekovski~\cite{CPS2022} studied
various properties of the odd chromatic number; in particular, they proved the following facts:
every outerplanar graph admits an odd $5$-coloring; every graph of maximum
degree three has an odd $4$-coloring; and for every connected planar graph $G$,
if   $|G|$ is even, or $|G|$ is odd and $G$ has a  vertex of degree $2$ or any odd
degree, then $\chi_o(G)\le8$ (a key step towards proving that $8$ colors
suffice for an odd coloring of any planar graph).   It is worth noting that
their proof of this key step relies on Theorem $4$ in Aashtab, Akbari,
Ghanbari, and Shidani \cite{4trees}, which itself relies on the Four-Color
Theorem~\cite{AppHak77, 4CT}.
Building on work of Caro, Petru\v{s}evski, and
\v{S}krekovski~\cite{CPS2022},  Petr and Portier~\cite{odd8} further
proved that 8 colors suffice for all planar graphs; that is, every planar
graph admits an odd $8$-coloring.   

In this paper, we focus on studying odd colorings of $1$-planar graphs,
where a graph is \Emph{1-planar} if it can be drawn in the plane so that each  
edge is crossed by at most one other edge. We prove the following main result.

\begin{thm} \label{t:mainthm}
Every $1$-planar graph admits an odd ${23}$-coloring.
\end{thm}

It seems non-trivial to prove a more general result for $k$-planar graphs
for all $k\ge1$.  
Using ideas from our proof of Theorem~\ref{t:mainthm}, we can prove that a
$1$-planar graph $G$ admits an odd $14$-coloring if every proper minor of $G$
is 1-planar.  However, in general the class
of $1$-planar graphs is not closed under edge-contraction~\cite{notclosed}.  
Currently, the best known lower bound comes from
the graph \Emph{$K_7^*$}, which is formed from the complete graph $K_7$ by
subdividing each edge exactly once.  As observed in
\cite{oddsparse,PetSkr22}, we have $\chi_o(K_7^*)=7$. 
As shown in Figure~\ref{f:k7star}, since $K_7$ is 2-planar, we can get
a 1-planar embedding of $K_7^*$ by placing each subdividing vertex between
the two crossings (when they exist) of its edge.
It is known~\cite{1planar} that every 1-planar graph is $7$-degenerate. 
It would be interesting to 
use this to prove that every
1-planar graph admits an odd $s$-coloring for some $s$ much closer to 7. 

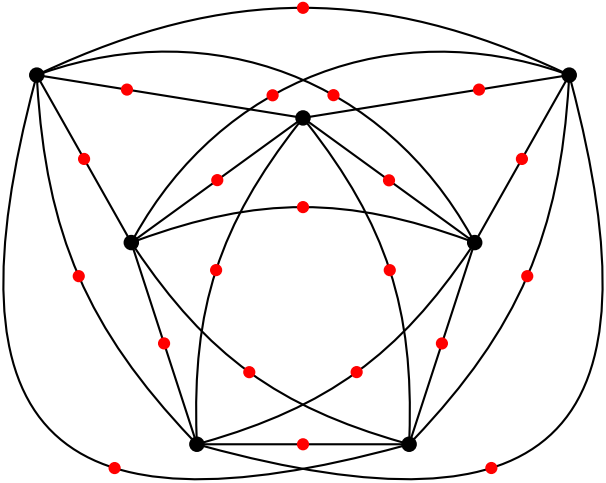
\begin{figure}[htb]
\centering
\begin{tikzpicture}[rotate=18, thick, scale=1.2]

\tikzset{1-thread/.style={postaction=decorate, decoration={markings,
    mark=at position 0.5 with {\node[usStyle] () {};} }}}
\tikzset{1-thread67/.style={postaction=decorate, decoration={markings,
    mark=at position 0.67 with {\node[usStyle] () {};} }}}
\tikzset{1-thread40/.style={postaction=decorate, decoration={markings,
    mark=at position 0.40 with {\node[usStyle] () {};} }}}
\tikzstyle{usStyle}=[shape = circle, minimum size = 4.0pt, inner sep = 0pt,
outer sep = 0pt, draw=red, fill=red]

\tikzstyle{uStyle}=[shape = circle, minimum size = 5.25pt, inner sep = 0pt,
outer sep = 0pt, draw, fill=black, semithick]
\tikzstyle{sStyle}=[shape = rectangle, minimum size = 4.5pt, inner sep = 0pt,
outer sep = 0pt, draw, fill=white, semithick]
\tikzstyle{lStyle}=[shape = circle, minimum size = 4.5pt, inner sep = 0pt,
outer sep = 0pt, draw=none, fill=none]
\tikzset{every node/.style=uStyle}

\foreach \i in {1,...,5}
\draw (\i*72:2cm) node[uStyle] (v\i) {};

\foreach \i/\j in {1/2, 2/3, 3/4, 4/5, 5/1}
\draw[1-thread] (v\i) -- (v\j);

\foreach \i/\j in {1/3, 2/4, 3/5, 4/1, 5/2}
\draw (v\i) edge[bend right=20, 1-thread] (v\j);

\draw (122:3.85cm) node[uStyle] (v6) {};
\draw (22:3.85cm) node[uStyle] (v7) {};

\draw[1-thread67] (v1) -- (v7);
\draw[1-thread] (v5) -- (v7);
\draw[1-thread67] (v1) -- (v6); 
\draw[1-thread] (v2) -- (v6); 
\draw (v6) edge[bend left=25, 1-thread] (v7); 
\draw (v3) edge[bend left=20, 1-thread] (v6); 
\draw (v4) edge[bend right=20, 1-thread] (v7); 
\draw (v2) edge[bend left=40, 1-thread40] (v7); 
\draw (v5) edge[bend right=40, 1-thread40] (v6); 
\draw (v4) edge[bend left=60, 1-thread40, looseness=1.8] (v6); 
\draw (v3) edge[bend right=60, 1-thread40, looseness=1.8] (v7); 
\end{tikzpicture}

\caption{A 1-plane embedding of $K_7^*$, where subdividing vertices are in red.}
\label{f:k7star}
\end{figure}

 We need to introduce more notation.  For a proper odd coloring $\tau$ of some
subgraph $J$ of $G$ and for each $v\in V(J)$, let \Emph{$\tau_o(v)$} 
denote the unique color that appears an odd number of times on $N_J(v)$ if such
a color exists; otherwise, $\tau_o(v)$ is undefined.  Define 
$\tau(N_J(v)):=\{\tau(u)\mid u\in N_J(v)\}$\aside{$\tau(N_J(v))$} and
$\tau_o(N_J(v)):=\{\tau_o(u)\mid u\in N_J(v)\}$.\aside{$\tau_o(N_J(v))$}  Note
that, in any process
that involves extending an odd coloring of $J$ to $G$  one vertex at a time,
the existence and identity of $\tau_o(v)$ may change each time a new vertex is
colored.   A \emph{$k$-vertex}\aside{$k$/$k^+$-vertex}  is a vertex of degree $k$.
A \emph{$k^+$-vertex}  is a vertex of degree at least $k$.    Let $G$ be  a
plane graph and let $f$ be a face of $G$.  We use $d(f)$ to denote the size of
$f$; in addition,   $f$ is a \emph{$k$-face}\aaside{$k$/$k^+$-face}{-6mm} of $G$ if
$d(f)=k$ and a \emph{$k^+$-face}  if $d(f)\ge k$.  For any positive integer
$k$, we let $[k]:=\{1,2, \ldots, k\}$. \aaside{$d(f)$, $[k]$}{-6mm}
A graph is \Emph{$d$-degenerate} if each subgraph $J$ has a vertex $v$ with
$d_J(v)\le d$.

We end this section with some easy results about graphs  with forbidden  minors. 
Let $\mathcal{F}$ denote a family of graphs that are closed under taking minors;
that is, if $G\in \mathcal{F}$, then every minor of $G$ also belongs to
$\mathcal{F}$.  We say $\mathcal{F}$ is \emph{$d$-degenerate} for some integer
$d\ge0$ if every graph in $\mathcal{F}$ is $d$-degenerate. As observed in
\cite{oddsparse,PetSkr22}, for every $n\in\mathbb{N}$, there exists a
$2$-degenerate graph $G$ with $\chi_o(G)=n$.  However, for  minor-closed
families, we prove the following.

\begin{thm} \label{t:minorclosed}
If $\mathcal{F}$ is minor-closed and $d$-degenerate, then $\chi_o(G)\le 2d+1$
for every $G\in\mathcal{F}$.
\end{thm}
\begin{proof} Suppose the statement is false. Let $G\in \mathcal{F}$ be a graph
with $\chi_o(G)\ge 2d+2$. We choose $G$ with $|G|$ minimum. Note that $G$ is
connected and $\delta(G)\le d$. Fix $xy\in E(G)$ with $d(x)=\delta(G)$.  Let
$G':= G/xy$ and let $w$ be the new vertex in $G'$; now $G'\in\mathcal{F}$.  
Since $G$ is minimal, $G'$ has an odd coloring $\tau:V(G')\rightarrow [2d+1]$.
We extend $\tau$ to $G$ by coloring $y$ with the  color $\tau(w)$ and coloring $x$ 
with a color in $[2d+1]\less (\tau(N(x))\cup \tau_o(N(x)))$; when coloring $x$,
note that we forbid at most $2d$ colors. It is simple to check that $\tau$ is an odd
$(2d+1)$-coloring of $G$, because $\tau(y)$ appears exactly once on $N(x)$. 
This is a contradiction.
\end{proof}

Corollary~\ref{t:K4} below follows directly from Theorem~\ref{t:minorclosed} and
the fact that every graph  with no $K_4$ minor is 2-degenerate; it extends the
result of Caro, Petru\v{s}evski, and \v{S}krekovski~\cite{CPS2022} on
outerplanar graphs. The sharpness of Corollary~\ref{t:K4} is witnessed by $C_5$.  

\begin{cor}\label{t:K4} 
Every graph with no $K_4$ minor admits an odd 5-coloring.
\end{cor}
 
For each integer $p$ with $5\le p\le 9$, it is known \cite{K8, Mader68, K9} 
that every graph with no $K_p$ minor is $(2p-5)$-degenerate; so $G$ has an odd
$(4p-9)$-coloring by Theorem~\ref{t:minorclosed}.

\section{Proof of Theorem~\ref{t:mainthm}}

In this section we prove our main result, Theorem~\ref{t:mainthm}: If $G$
is a 1-planar graph, then $\chi_o(G)\le {23}$.\medskip

Suppose the statement is false, and let $G$ be a counterexample minimizing $|G|$. 
Note that $G$ is connected and $\chio(G)\ge {24}$.  We consider a plane embedding of
$G$ with as few edge crossings as possible.  Throughout the proof,
we use $N(v)$ and $d(v)$\aside{$N(v)$, $d(v)$} to denote the neighborhood and degree of a vertex $v$
in $G$, respectively. We next prove a sequence of 8 claims. 

\begin{clm}\label{2ec}
$G$ is 2-edge-connected and so $\delta(G)\ge2$.
\end{clm}

\begin{proof} Suppose instead that   $xy\in E(G)$ is a bridge in $G$.
Let $G_1, G_2$ be the components of $G \setminus xy$ such that   $x\in V(G_1)$ and $y\in V(G_2)$.
By the choice of $G$, let $\tau_1:V(G_1)\rightarrow [{23}]$ be an odd  coloring
of $G_1$  such that $\tau_1(x)=1$ and color $2$ appears an odd number of times
on $N(x)\less y$ if  $x$ is a $2^+$-vertex; let  $\tau_2:V(G_2)\rightarrow
[{23}]$  be an odd  coloring of $G_2$  such that $\tau_2(y)=3$ and color $4$
appears an odd number of times on $N(y)\less x$ if $y$ is a $2^+$-vertex.  It
is simple to see that we can combine $\tau_1$ and $\tau_2$ into an odd
{23}-coloring of $G$, a contradiction. Thus $G$ is 2-edge-connected and so $\delta(G)\ge2$.
\end{proof}

\begin{clm}\label{oddvertex}
Every odd vertex in $G$ has degree at least {13}.
\end{clm}
 
\begin{proof}
Suppose not, and let $v \in V(G)$ be an odd vertex with $d(v)\le
{11}$.
Now $G \setminus v$ has an odd {23}-coloring $\tau$ by  the minimality of $G$.
But then $\tau$ can be extended to an odd {23}-coloring of $G$ by assigning to
$v$ a color in $[{23}]\setminus(\tau(N(v))\cup \tau_o(N(v)))$, since 
at most {22} colors are forbidden. (Note that some color must 
appear an odd number of times on  $N(v)$, since $d(v)$ is odd.)
\end{proof}

For the rest of the proof,  a vertex $v$ in $G$ is \emph{big} if $d(v)\ge
{12}$; 
otherwise $v$ is \emph{small}.\aside{big, small}  

\begin{clm}\label{c:bign}
No two small vertices are adjacent in $G$.
\end{clm}

\begin{proof} Suppose not, and let $v$ and $w$ be two adjacent small vertices in $G$.
By Claim~\ref{oddvertex},  $d(v)\le {10}$ and $d(w)\le
{10}$.  Let $G':=G \setminus
\{ v, w \}$ and let $\tau$ be an odd {23}-coloring of $G'$.
Now $|\tau(N_G'(x))\cup \tau_o(N_G'(x))|\le {20}$ for each $x\in \{v,w\}$. 
But $\tau$ can be extended to an odd {23}-coloring of $G$ as follows: 
 first assign $v$ a color in $[{23}]\setminus(\{ \tau_o(w) \} \cup
\tau(N_G'(v))\cup \tau_o(N_G'(v)))$; then  assign $w$ a color in
$[{23}]\setminus(\{ \tau(v), \tau_o(v) \} \cup \tau(N_{G'}(w))\cup \tau_o(N_{G'}(w)))$.
\end{proof}

\begin{clm}\label{c:crossing}
Every edge incident to a small vertex in $G$  has a crossing.
\end{clm}
 
\begin{proof}
Suppose instead that  there exists  $xy \in E(G)$ such that $x$ is a
small vertex and $xy$ has no crossing.  Form $G'$ from $G$
by deleting vertex $x$ and adding an edge $yz$ for each $z\in
N(x) \setminus N(y)$; since edge $xy$ has no crossing, we can guarantee
that each new edge $yz$  has  a crossing if and only if $xz$ has a crossing. 
Thus $G'$ is 1-planar.  By the minimality of $G$, let $\tau$ be an odd
{23}-coloring of $G'$.  We can extend $\tau$ to an odd
{23}-coloring  of
$G$ by coloring $x$ from $[{23}]\setminus(\tau(N(x))\cup 
\tau_o(N(x)))$; at least {1} color remains available, and $\tau(y)$ 
occurs exactly once on $N(x)$, a contradiction.  
\end{proof} 

For each vertex $v$, let \Emph{$d_2(v)$} denote the number of 2-vertices adjacent to
$v$ in $G$.  The following claim is a simplified version of
Lemma~2.1~in~\cite{CCKP}.

{
\begin{clm}\label{borrowed-lem}
If $v$ is a vertex with $d_2(v)\ge 1$, then $2d(v)\ge d_2(v)+23$.
\end{clm}
\begin{proof}
Suppose the contrary.  Form $G'$ from $G$ by deleting $v$ and all adjacent
2-vertices.  By the minimality of $G$, we know that $G'$ has an odd 23-coloring
$\tau$.  We color $v$ to avoid the color on each $3^+$-vertex $w$ that is adjacent
in $G$, as well as the color on $\tau_o(w)$ for each adjacent $3^+$-vertex $w$,
and the color on the other neighbor of each adjacent 2-vertex (which is deleted in $G'$).
By assumption, $2d(v)-d_2(v)\le 22$, so we can color $v$.  Now for each 2-vertex
$w$ adjacent (in $G$) to $v$, color $w$ to avoid the colors on its two
neighbors, say $v$ and $x$, in $G$, as well as $\tau_o(v)$ and $\tau_o(x)$.
It is easy to check that this gives an odd 23-coloring of $G$, a contradiction.
\end{proof}
}

Let \Emph{$H$} be the plane graph formed from $G$ by replacing each crossing
with a ``virtual'' 4-vertex.  Now $H$ is 2-edge-connected by Claim~\ref{2ec}.
Note that no two virtual 4-vertices are adjacent in $H$.  Moreover, by
Claim~\ref{c:crossing},  $N_H(v)$ consists entirely of virtual 4-vertices for
each small vertex $v$ in $G$. 

\begin{clm}\label{3-face}
The graph $H$ has no loop or 2-face.
Every 3-face in $H$ is incident to either three big vertices or two big
vertices and one virtual 4-vertex.
\end{clm} 

\begin{proof} 
Clearly, $H$ cannot have any loops, since $G$ is loopless.  Suppose that $H$
contains a 2-face, with boundary $vw$.  If $v$ and $w$ are both vertices of $G$,
then $G$ contains parallel edges, a contradiction.  As noted above, every two
virtual 4-vertices of $H$ are non-adjacent.  So assume that $v\in V(G)$ and $w$
is a virtual 4-vertex.  Since $vw$ is a 2-face in $H$, vertex $w$ arises in $H$
from a crossing of two edges in $G$ that are both incident to $v$.  However, we
can ``uncross'' these edges, somewhat akin to a simpler version of what we show
in Figure~\ref{f:sixfour}, yielding an embedding of $G$ with fewer edge crossings.
This contradicts our choice of $G$, which proves the first statement.

Let $v$ be a small vertex in $G$.  As observed above,  every edge incident to
$v$ has a crossing by Claim~\ref{c:crossing}.  So every neighbor
of $v$ in $H$ is a virtual 4-vertex.  Thus, all neighbors of $v$ are
pairwise non-adjacent in $H$; so $v$ is not incident to any 3-face in $H$.
This proves the second statement.
\end{proof}

\begin{clm}\label{c:5face}
Every 2-vertex in $H$ is incident to a $5^+$-face and to another $4^+$-face.
\end{clm}
 
\begin{proof} 
Let $x$ be a 2-vertex in $H$ and let $y$ and $z$ be the neighbors of $x$ in $H$.
By Claim~\ref{c:crossing},   both $y$ and $z$ are virtual 4-vertices. Recall
that $H$ is 2-edge-connected.   Let $f_1$ and $f_2$ be the two distinct faces
that contain the  vertex $x$.  So $x, y, z\in V(f_1)\cap V(f_2)$.  Since
$yz\notin E(H)$, both $f_1$ and $f_2$ must be $4^+$-faces.  Suppose
both $f_1$ and $f_2$ are 4-faces.  Let $u, w$ be the remaining vertices of $f_1$
and $f_2$, respectively.  Now each vertex in $\{u, w\}$ is adjacent in $H$ to
both $y$ and $z$.  But this implies that $G$ has a multiple edge between $u$ and
$w$, a contradiction.
\end{proof}

We use discharging on $H$ to reach a contradiction. Let
\Emph{$F(H)$} denote the set of all faces of $H$. For each vertex $v\in V(H)$
and each face $f\in{F(H)}$,   let  $\ch(v):= d_{H}(v)-4$\aside{$\ch(v)$} be the initial
charge of $v$ and $\ch(f):= d(f)-4$\aside{$\ch(f)$} be the initial charge of $f$.   By 
Euler's formula, the total charge of $H$ is  
\begin{align*}
\ch(H):&=\sum_{v \in V(H)} (d(v) - 4) + \sum_{f \in F(H)} (d(f) - 4) \\
&= -4( |H| - e(H) + |F(H)|)\\
&= -8.
\end{align*}

Note that $\ch(v)<0$ if and only if $d_H(v)=2$; similarly, $\ch(f)<0$ if and only if
$f$ is a 3-face.  We will redistribute the charge on $H$ according to the
following three discharging rules. 
 
\begin{itemize}
\item[(R1)] Every $5^+$-face  splits its charge equally among all incident 2-vertices.  
\item[(R2)] Every big vertex sends $\frac{1}{2}$ to each incident 3-face.
\item[(R3)] Every big vertex  $v$ sends $\frac{1}{2}$ to each 2-vertex in
$N_{{G}}(v)$.
\end{itemize}

Let $\ch^*$\aside{$\ch^*$} denote the new charge, after applying (R1)--(R3).
Note that $\ch^*(H)=\ch(H)=-8$.  We reach a contradiction by showing that
$\ch^*(v)\ge0$  and $\ch^*(f)\ge0$ for every vertex $v\in V(H)$ and every face
$f\in{F(H)}$. 
Every vertex $v$ that is small but not a 2-vertex has $\ch^*(v)=d(v)-4\ge 0$.
Every 4-face $f$ has $\ch^*(f)=d(f)-4=0$.  And every $5^+$-face $f$ has
$\ch(f)=d(f)-4>0$, so $\ch^*(f)\ge 0$, by (R1).
By Claim~\ref{3-face}, every 3-face $f$ is incident to at least two big
vertices.  So by (R2),  $\ch^*(f)\ge\ch(f)+ \frac12+ \frac12 = 0$. 
Thus, it suffices to consider only 2-vertices and big vertices.

\begin{figure}[htb]
\centering
\begin{tikzpicture}[thick, rotate=0]

\tikzstyle{redV}=[shape = circle, minimum size = 4.0pt, inner sep = 0pt,
outer sep = 0pt, draw=red, fill=red]

\tikzstyle{blackV}=[shape = circle, minimum size = 5.25pt, inner sep = 0pt,
outer sep = 0pt, draw, fill=black, semithick]
\tikzstyle{sStyle}=[shape = rectangle, minimum size = 4.5pt, inner sep = 0pt,
outer sep = 0pt, draw, fill=white, semithick]
\tikzstyle{lStyle}=[shape = circle, minimum size = 4.5pt, inner sep = 0pt,
outer sep = 0pt, draw=none, fill=none]
\tikzset{every node/.style=blackV}

\begin{scope}[rotate=10]
\foreach \i/\type in {1/redV, 2/blackV, 3/redV, 4/blackV, 5/redV, 6/blackV,
7/redV, 8/blackV, 9/redV, 10/blackV, 11/redV, 12/blackV, 13/redV,
14/blackV, 15/redV, 16/blackV, 17/redV, 18/blackV}
\draw (\i*20:2cm) node[\type] (v\i) {};

\foreach \i in {1,...,18}
\draw (0,0) -- (v\i);

\foreach \i/\j in {18/2, 2/4, 4/6, 6/8, 8/10, 10/12, 12/14, 14/16, 16/18}
\draw (v\i) edge[bend left=20] (v\j);

\draw (0,0) node[blackV] {};
\end{scope}

\begin{scope}[xshift=-2.5in, rotate=-76]

\foreach \i/\type in {0/blackV, 1/redV, 2/blackV, 3/redV, 4/blackV, 5/redV,
6/blackV, 7/redV, 8/blackV, 9/redV, 10/blackV, 11/redV, 12/blackV, 13/redV,
14/blackV, 15/redV, 16/blackV, 17/redV, 18/blackV}
\draw (\i*18.5:2cm) node[\type] (v\i) {};

\foreach \i in {1,...,17}
\draw (0,0) -- (v\i);

\foreach \i/\j in {0/2, 2/4, 4/6, 6/8, 8/10, 10/12, 12/14, 14/16, 16/18}
\draw (v\i) edge[bend left=20] (v\j);

\draw (0,0) node[blackV] {};

\end{scope}
\end{tikzpicture}

\caption{A 17-vertex and 18-vertex giving away maximum charge; red vertices
are 2-vertices.
\label{17/18-vertex-fig}}
\end{figure}
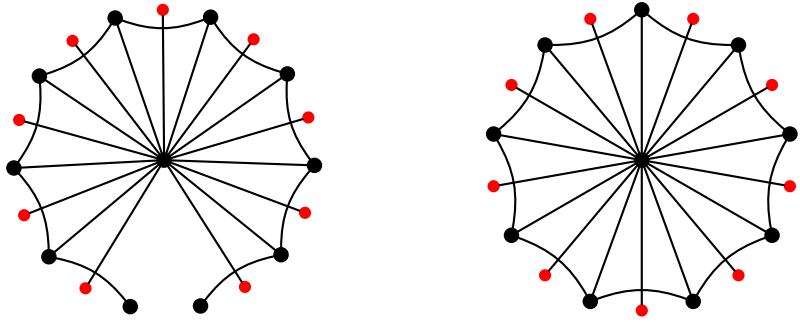
\bigskip

\begin{clm}\label{big}
If $v$ is a big vertex, then $\ch^*(v) \geq 0$.
\end{clm}

\begin{proof} 
Let $v$ be a big vertex.  
{
If $d(v)\le 15$, then Claim~\ref{borrowed-lem} implies that $d_2(v)\le
2d(v)-23$.  So $\ch^*(v)\ge d(v)-4-d(v)/2-d_2(v)/2\ge
d(v)/2-4-(d(v)-23/2)=15/2-d(v)/2\ge 0$.
Now assume instead that $d(v)\ge 16$.
}
Figure~\ref{17/18-vertex-fig} shows two examples.
To bound the total charge that $v$ gives away by (R2)
and (R3), we simulate these rules as follows: $v$ gives $3/4$ to each
neighbor $w$ in $H$; if $w$ is not a virtual 4-vertex, then $w$ gives $3/8$ to
each 3-face $f$ in $H$ that is incident with both $v$ and $w$; if $w$ is a
virtual 4-vertex, then $w$ gives $1/8$ to each 3-face in $H$ that is
incident with both $v$ and $w$, and gives $1/2$ to each 2-vertex adjacent to $v$
that helped form $w$.  It is easy to check that each adjacent 2-vertex and incident
3-face get at least as much charge from $v$ with these simulated rules as they
do via (R2) and (R3).  Thus, $\ch^*(v)\ge d(v)-4-3d(v)/4\ge 0$, since $d(v)\ge 16$.
\end{proof}

\begin{clm} 
If $v$ is a 2-vertex, then $\ch^*(v) \geq 0$.
\end{clm}

\begin{proof}  
Let $v$ be a $2$-vertex in $H$.  Since $H$ is 2-edge-connected, let  $f,
f'\in{F(H)}$ be the two faces incident to $v$.    Let $a, b$ be
the neighbors of $v$ in $G$.  Figure~\ref{f:sixfour} shows a special case.
By Claim~\ref{c:bign},  both $a$ and $b$ are big
vertices. By (R3), each of $a$ and $b$ sends  a $\frac{1}{2}$ to
$v$, and $\ch(v) = -2$, so we need to show that $v$ receives a total 
of at least 1 from $f$ and $f'$.
By Claim~\ref{c:5face},  we may assume that $f$ is a $5^+$-face and $f'$ is a
$4^+$-face.  If $f$ is a 5-face, then $f$ is incident to at most one $2$-vertex
(since every neighbor of each 2-vertex is a virtual $4$-vertex and no two
virtual $4$-vertices are adjacent), so $f$ sends $1$ to $v$ by 
(R1) and we are done.
If $f$ is a 7-face, then $f$ is incident to at most two $2$-vertices, so $f$
sends $v$ at least $\frac{3}{2}$, and we are done.
If $f$ is an $8^+$-face, then $f$ is incident to at most $\left\lfloor
\frac{d(f)}{2} \right\rfloor$ $2$-vertices, and   
\[ \ch(f)=d(f) - 4 \geq
d(f) - \left\lceil \frac{d(f)}{2} \right\rceil \geq \left\lfloor \frac{d(f)}{2}
\right\rfloor, \] so $f$ sends at least 1 to $v$ and we are done.
Thus, we assume that $d(f) = 6$.  

Recall that $f'$ is a $4^+$-face.
If $f'$ is a $5$-face or a $7^+$-face, then, as before, we are done;  and if
$f'$ is a $6$-face, then each of $f$ and $f'$ sends at least
$\frac{2}{3}$ to $v$, for a total of $\frac{4}{3}$.
So we assume that $f'$ is a $4$-face.  We claim that $f$ is incident
to at most two $2$-vertices.  Suppose $f$ is incident to three $2$-vertices.
Let $u$ and $w$ be the remaining $2$-vertices on $f$, let  $z$ be the common
neighbor of $u$ and $w$ on $f$, and let $c$ be the vertex on $f'$ that is
neither $v$ nor a virtual $4$-vertex. See Figure~\ref{f:sixfour}. Note that
$z$ is a virtual $4$-vertex.  
Now $u$ and $w$ are each adjacent in $G$ to $c$, with edges $uc$ and
$wc$ each crossing an edge incident to $v$, with the remaining edges
incident to $u$ and $w$ crossing each other.
Now by switching $u$ and $w$, we create a $1$-planar embedding of $G$
with fewer crossings, contradicting our choice of the embedding of $G$.
Thus, $f$ has at most two incident $2$-vertices, as claimed.  So $f$ sends $v$
at least $1$.  Hence, $\ch^*(v)\ge 0$, as desired.
\end{proof}

\begin{figure}[htb]

\centering
\begin{tikzpicture}[thick, scale=.7]

\tikzstyle{blackV}=[shape = circle, minimum size = 6.25pt, inner sep = 0pt,
outer sep = 0pt, draw, fill=black, semithick]
\tikzstyle{redV}=[shape = circle, minimum size = 4.25pt, inner sep = 0pt,
outer sep = 0pt, draw=red, fill=red]
\tikzstyle{blueV}=[shape = circle, minimum size = 5.5pt, inner sep = 0pt,
outer sep = 0pt, draw=blue!99!black, fill=blue!99!black]
\tikzstyle{uStyle}=[shape = rectangle, minimum size = 4.5pt, inner sep = 0pt,
outer sep = 0pt, draw=none, fill=none]

\def\off{4mm}

\begin{scope}
\foreach \i in {1,2}
\draw (\i*160-150:1.0cm) node[redV] (v\i) {};

\draw (270:1.0cm) node[redV] (v3) {};

\foreach \i/\ang in {4/-10,5/190,6/270}
\draw (\ang:2.65cm) node[blackV] (v\i) {};

\foreach \i/\ang in {7/115,8/65}
\draw (\ang:2.325cm) node[blackV] (v\i) {};

\draw (v8) -- (v2) -- (v6) -- (v1) -- (v7);
\draw (v4) -- (v3) -- (v5);

\draw (90:1.170cm) node[blueV] (v9) {};
\draw (235:1.05cm) node[blueV] {};
\draw (305:1.05cm) node[blueV] {};

\foreach \i in {1,2,3}
\draw (v\i) node[redV] {};

\foreach \i/\name in {1/w,2/u,3/v,4/b,5/a,9/z}
\draw (v\i) ++(0,\off) node[uStyle] {\footnotesize{$\name$}};

\draw (v6) ++(0,-\off) node[uStyle] {\footnotesize{$c$}};

\draw (90:.25cm) node[uStyle] {\footnotesize{$f$}};
\draw (270:1.65cm) node[uStyle] {\footnotesize{$~f'$}};
\end{scope}

\begin{scope}[xshift=3in]
\foreach \i in {1,2}
\draw (\i*160-150:1.0cm) node[redV] (v\i) {};

\draw (270:1.0cm) node[redV] (v3) {};

\foreach \i/\ang in {4/-10,5/190,6/270}
\draw (\ang:2.65cm) node[blackV] (v\i) {};

\foreach \i/\ang in {7/115,8/65}
\draw (\ang:2.325cm) node[blackV] (v\i) {};

\draw (v8) -- (v2) -- (v6) -- (v1) -- (v7);
\draw (v4) -- (v3) -- (v5);

\foreach \i in {1,2,3}
\draw (v\i) node[redV] {};

\foreach \i/\name in {1/w,2/u,3/v,4/b,5/a,9/z}
\draw (v\i) ++(0,\off) node[uStyle] {\footnotesize{$\name$}};

\draw (v6) ++(0,-\off) node[uStyle] {\footnotesize{$c$}};
\end{scope}

\begin{scope}[xshift=6in]
\foreach \i in {1,2}
\draw (\i*160-150:1.0cm) node[redV] (v\i) {};

\draw (270:1.0cm) node[redV] (v3) {};

\foreach \i/\ang in {4/-10,5/190,6/270}
\draw (\ang:2.65cm) node[blackV] (v\i) {};

\foreach \i/\ang in {7/115,8/65}
\draw (\ang:2.325cm) node[blackV] (v\i) {};

\draw (v2) -- (v6) -- (v1);
\draw (v8) edge[bend right=60, looseness=1.15] (v1);
\draw (v7) edge[bend left=60, looseness=1.15] (v2);
\draw (v4) -- (v3) -- (v5);

\foreach \i in {1,2,3}
\draw (v\i) node[redV] {};

\foreach \i/\name in {1/u,2/w,3/v,4/b,5/a,9/z}
\draw (v\i) ++(0,\off) node[uStyle] {\footnotesize{$\name$}};

\draw (v6) ++(0,-\off) node[uStyle] {\footnotesize{$c$}};
\end{scope}

\end{tikzpicture}

\caption{From left to right: the vertex $v$, its incident 4-face and 6-face in
$H$, and their incident vertices; the same configuration
in $G$; a 1-plane embedding of $G$ with fewer crossings.
\label{f:sixfour}
}
\end{figure}
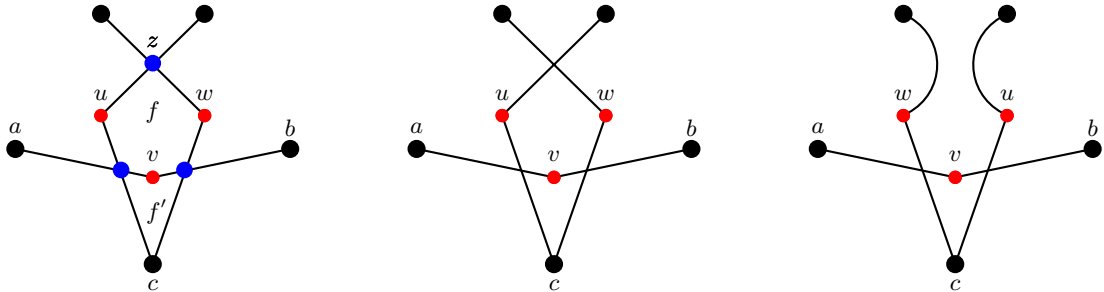

We have shown that every vertex and face of $H$ finishes with nonnegative charge. 
So $-8=\ch(H)=\ch^*(H)\ge 0$, which is a contradiction.
\qed

\small{

}


\begin{thebibliography}{99}
\bibitem{4trees} 
A.~Aashtab, S.~Akbari, M.~Ghanbari, and A.~Shidani,  Vertex
Partitioning of Graphs into Odd Induced Subgraphs, Discuss. Math. Graph Theory. 
%doi:10.7151/dgmt.2371
    
\bibitem{AppHak77} 
K.~Appel and W.~Haken,  The Solution of the Four-Color Map Problem, Sci. Amer.  237 (1977) 108--121.
    
\bibitem{CPS2022} 
Y. Caro, M.~Petru\v{s}evski, and R.~\v{S}krekovski, Remarks
on Odd Colorings of Graphs, 	
\href{http://arxiv.org/abs/2201.03608}{\texttt{arXiv:2201.03608}}.
    
\bibitem{notclosed} Z.-Z. Chen and M.Kouno, A Linear-time Algorithm for
$7$-coloring $1$-plane Graphs, Algorithmica 43 (2005) 147--177.

\bibitem{CCKP}
E.-K.~Cho, I.~Choi, H.~Kwon, and B.~Park,
Odd Coloring of Sparse Graphs and Planar Graphs,
\href{http://arxiv.org/abs/2202.11267}{\texttt{arXiv:2202.11267}}.

\bibitem{oddsparse} 
D.~W.~Cranston, Odd Colorings of Sparse Graphs, 	
\href{http://arxiv.org/abs/2201.01455}{\texttt{arXiv:2201.01455}}.
    
\bibitem{K8} 
L. K. J{\o}rgensen,  Contractions to $K_8$,  J. Graph Theory  18 (1994) 431--448.  
    
\bibitem{Mader68} 
W.~Mader,  Homomorphies\"atze f\"ur {G}raphen, 
 Math. Ann. 178 (1968) 54--168.
     
\bibitem{1planar} 
J. Pach and G.  T\'oth, Graphs Drawn with few Crossings per
Edge,  Combinatorica 17 (1997) 427--439.
      
\bibitem{odd8}
J. Petr and J. Portier,  Odd Chromatic Number of Planar Graphs
is at most $8$,	
\href{http://arxiv.org/abs/2201.12381}{\texttt{arXiv:2201.12381}}.
    
\bibitem{PetSkr22} 
M.~Petru\v{s}evski, and R.~\v{S}krekovski, Colorings with
Neighborhood Parity Condition,
\href{http://arxiv.org/abs/2112.13710}{\texttt{arXiv:2112.13710}}.
    
\bibitem{4CT} 
N.~Robertson, D.~P.~Sanders, P.~D.~Seymour, and R.~Thomas,  A New Proof of the
Four Colour Theorem, Electron. Res. Announc. Amer. Math. Soc. 2 (1996) 17--25.    

\bibitem{K9} 
Z.-X. Song and  R. Thomas,  The Extremal Function for $K_9$ Minors, 
 J. Combin. Theory Ser. B 96 (2006) 240--252.
\end{thebibliography}
\end{document}